\documentclass[11pt,a4paper]{article}
 \usepackage{euscript}
\usepackage{amsmath}
\usepackage{graphicx}
\usepackage{amsthm}
\usepackage{amssymb}
\usepackage{epsfig}
\usepackage{url}
\usepackage{colonequals}
\usepackage{amsmath,amscd}
\theoremstyle{theorem}
\newtheorem{theorem}{Theorem}[section]
\newtheorem{corollary}[theorem]{Corollary}

\newtheorem{example}{Example}[section]

\numberwithin{equation}{section}

\title{Hilbert's Nullstellensatz over quaternions} 
\author{Masood Aryapoor\\
\tiny{\textit{Division of Mathematics and Physics}}\\
\tiny{\textit{M\"{a}lardalen  University}}\\
\tiny{\textit{Hamngatan 15, 632 17, Eskilstuna, 
Sweden
}}
}
 \date{}
\begin{document}
 \maketitle
\begin{abstract}
\noindent
Using the Rabinowitsch trick, we prove a version of Nullstellensatz over quaternions, which generalizes  Hilbert's Nullstellensatz  over complex numbers. 
\end{abstract}
\begin{section}{Introduction}

In \cite{AlonNullestellensatz}, two forms of  (central) Nullstellensatz over the division ring of quaternions $\mathbb{H}$ are presented: a weak form and a strong form. The authors of \cite{AlonNullestellensatz} remark that ``the classical strong Nullstellensatz for $\mathbb{C}[x_1,\dots, x_n]$ can be easily deduced as an immediate consequence of the weak Nullstellensatz, using the famous Rabinowitsch trick. Such a proof does not seem possible for $\mathbb{H}[x_1,\dots, x_n]$, since substitution is not a homomorphism.'' Here, $\mathbb{H}[x_1,\dots, x_n]$ denotes the ring of polynomials over $\mathbb{H}$ in $n$ central indeterminates. Our goal is to fill the gaps by adapting the Rabinowitsch trick to derive a version of Nullstellensatz over quaternions that generalizes the following form of Hilbert's Nullstellensatz over complex numbers:
\begin{theorem}[Hilbert's Nullstellensatz over complex numbers]
	A polynomial $F\in \mathbb{C}[x_1,\dots, x_n]$ vanishes on the zero locus of an ideal $I$  of  $\mathbb{C}[x_1,\dots, x_n]$  iff there exists a natural number $N\geq 1$ such that $F^N\in I$. 
\end{theorem}
More precisely, we shall show the following result, the proof of which is given in Subsection \ref{sec: Strong Nullstellensatz over quaternions}:
\begin{theorem}[Hilbert's Nullstellensatz over quaternions]\label{thm: Strong Nullstellensatz over quaternions}
	A polynomial $F\in \mathbb{H}[x_1,\dots, x_n]$ vanishes on the zero locus of a left ideal $I$ of $\mathbb{H}[x_1,\dots, x_n]$  iff for every $a\in\mathbb{H}$, there exists a natural number $N\geq 1$ such that 
	\begin{equation}\label{eq: aF}
		(aF)^N\in I+I(aF)+I(aF)^2+\cdots+I(aF)^N.
	\end{equation}
\end{theorem}
Our proof of the theorem is similar to deriving Hilbert's Nullstellensatz over complex numbers from the weak form of Nullstellensatz, using the Rabinowitsch trick. The main tool used in the proof is the theory of evaluation of polynomials in several central variables over division rings, which is developed in \cite{bennenni2023evaluation}.
   
\end{section} 


\begin{section}{Hilbert's Nullstellensatz over quaternions}
	
In this section, we prove Hilbert's Nullstellensatz over quaternions.


\begin{subsection}{Evaluation of polynomials over division rings}
	This subsection deals with some facts regarding evaluation of polynomials over division rings. For more details, see \cite{bennenni2023evaluation}.  Let $K$ be a division ring and $K[x_1,\dots,x_n]$ be the ring of polynomials over $K$ in central indeterminates $x_1,\dots,x_n$. The set of points $(a_1,...,a_n)\in K$ for which $a_ra_s=a_sa_r$ for all $r,s$ is denoted by $K^n_c$. One can show that for every point $p=(a_1,\dots,a_n)\in  K^n_c$ and polynomial $F\in K[x_1,\dots,x_n]$, there exists a unique element $F(p)=F(a_1,\dots,a_n)\in K$, called the evaluation of $F$ at $p$, such that the polynomial $F-F(p)$ belongs to the left ideal of $K[x_1,\dots,x_n]$ generated by the polynomials $x_1-a_1,...,x_n-a_n$, that is,
	\[
		F-F(p)\in \sum_{m=1}^n\mathbb{H}[x_1,\dots,x_n]\,(x_m-a_m). 
	\]
	 We also have the product formula: for all points $p=(a_1,\dots,a_n)\in  K^n_c$  and polynomials $F,G\in K[x_1,\dots,x_n]$, 
	\begin{equation}\label{equ: product formula}
		(FG)(p)=
		\begin{cases}
			0& \text{ if } G(p)=0,\\
			F\left( G(p)pG(p)^{-1} \right) & \text{ if  } G(p)\neq 0.
		\end{cases}
	\end{equation}
	Here, the notation $bpb^{-1}$, where $(a_1,\dots,a_n)\in K^n$ and $b\in K\setminus\{0\}$, stands for the point
	\[
		bpb^{-1}\colonequals (ba_1b^{-1},\dots,ba_nb^{-1} )\in K^n.
	\]
 	Note that  if $p\in K^n_c$ and $b\in K\setminus\{0\}$, then $bpb^{-1}\in K^n_c$.
	  
\end{subsection}

\begin{subsection}{The weak form of Nullstellensatz over quaternions}
	
In this part, we review the weak form of Nullstellensatz over quaternions,  which is given in \cite{AlonNullestellensatz}. Let $\mathbb{H}$ denote the division ring of quaternions. Given a left ideal $I$ of $\mathbb{H}[x_1,\dots,x_n]$, we define the zero locus of $I$ to be the subset $\mathcal{V}(I)$ of $\mathbb{H}^n_c$ that is defined by 
\[
	\mathcal{V}(I)=\{p\in \mathbb{H}^n_c\,|\, F(p)=0 \text{ for all } F\in I\}. 
\]
For a subset $A$ of  $\mathbb{H}^n_c$, the set of polynomials $F\in \mathbb{H}[x_1,\dots,x_n]$ satisfying $F(p)=0$ for all $p\in A$ is denoted by $\mathcal{I}(A)$. As an easy consequence of the product formula  \ref{equ: product formula}, we see that  $\mathcal{I}(A)$ is always a left ideal of $ \mathbb{H}[x_1,\dots,x_n]$. The result below is the weak form of Nullstellensatz over quaternions (for a proof, see \cite[Thoerem 1.1]{AlonNullestellensatz}):
\begin{theorem}[Weak Nullstellensatz over quaternions]\label{thm: Weak Nullstellensatz over quaternions}
	For all left ideals $I$ of $\mathbb{H}[x_1,\dots,x_n]$, the set $\mathcal{V}(I)$ is empty iff $I=\mathbb{H}[x_1,\dots,x_n]$. 
\end{theorem}
\end{subsection}

	\begin{subsection}{The proof of Hilbert's Nullstellensatz over quaternions using the Rabinowitsch trick}\label{sec: Strong Nullstellensatz over quaternions}
		In this part, we give a proof of Theorem \ref{thm: Strong Nullstellensatz over quaternions}. 
	\begin{proof}[Proof of Theorem \ref{thm: Strong Nullstellensatz over quaternions}] 
		First, assume that the polynomial $F$ vanishes on the zero locus of $I$, that is, $F\in \mathcal{I}(\mathcal{V}(I))$.   Consider the ring of polynomials $\mathbb{H}[x_1,\dots,x_n][y]$ over $\mathbb{H}[x_1,\dots,x_n]$ in a central indeterminate $y$. I claim that the left ideal of $\mathbb{H}[x_1,\dots,x_n][y]$ which is generated by $I$ and the polynomial 
		$
			Fy-1
		$
		is equal to  $\mathbb{H}[x_1,\dots,x_n][y]$. To prove the claim, assume the contrary, i.e.,  this left ideal, denoted by $J$, is proper. It follows form the weak Nullstellensatz  (Theorem \ref{thm: Weak Nullstellensatz over quaternions}) applied to the polynomial ring $\mathbb{H}[x_1,\dots,x_n,y]$ that there exists a point	
		\[
			p=(a_1,\dots,a_n,b)\in \mathbb{H}^{n+1}_c
		\]
		such that $G(p)=0$ for all $G\in J$. Since $I\subset J$, we have $(a_1,\dots,a_n)\in \mathcal{V}(I)$. Evaluating $Fy-1\in J$ at $p$ and using the fact that $F\in \mathcal{I}(\mathcal{V}(I))$, we arrive at the contradiction $-1=0$. This completes the proof of the claim. As a consequence, we obtain an identity of the form 
		\[
			G\left(	Fy-1 \right)+\sum_{m= 0}^NG_{m}y^{m}=1,
		\]
		for some $G\in \mathbb{H}[x_1,\dots,x_n][y]$ and $G_{0},\dots,G_N\in I$.  Evaluating both sides of this identity  at the point $F^{-1}\in \mathbb{H}(x_1,\dots,x_n)$, where $\mathbb{H}(x_1,\dots,x_n)$ is the (classical) field of fractions of $\mathbb{H}[x_1,\dots,x_n]$,  we obtain an identity of the form 
		\[
			\sum_{m=0}^NG_{m}F^{-m}=1.
		\]
 We multiply both sides of this identity by $F^N$ (on the right), yielding
		\[
		F^N= \sum_{m=0}^NG_{m}F^{N-m}\in  I+IF+IF^2+\cdots+IF^N. 
		\]
		Since $aF\in \mathcal{I}(\mathcal{V}(I))$ for all $a\in \mathbb{H}$, a similar argument gives the desired condition for $aF$, completing the proof of this direction.
		
		To prove the converse, let $F$ satisfy Condition \ref{eq: aF}. We need to show that $F(a_1,...,a_n)=0$ for all $(a_1,...,a_n)\in \mathcal{V}(I)$. Assume, on the contrary, that there exists  $(a_1,...,a_n)\in \mathcal{V}(I)$ such that $F(a_1,...,a_n)\neq 0$. Set $a=F(a_1,...,a_n)^{-1}$. It is easy to verify (using the product formula \ref{equ: product formula}) that $(H(aF))(a_1,...,a_n)=H(a_1,...,a_n)$ for all $H\in \mathbb{H}[x_1,\dots,x_n]$.  Applying Condition \ref{eq: aF} for $a=F(a_1,...,a_n)^{-1}$, we can find polynomials  $G_0,...,G_N\in I$ such that 
		\[
			(aF)^N= G_{0}+G_{1}(aF)+\cdots+G_{N}(aF)^N. 
		\]
		Evaluating the sides of this equality at $(a_1,...,a_n)$ leads to the contradiction $1=0$, which completes the proof of the theorem. 
	\end{proof}
			
	\end{subsection}
	\begin{subsection}{Some comments} 
		We conclude the paper with some remarks. Our first remark concerns Condition \ref{eq: aF} in Theorem \ref{thm: Strong Nullstellensatz over quaternions}. It turns out that the condition  can be somewhat weakened, as follows: 
		\begin{corollary}
			A polynomial $F\in \mathbb{H}[x_1,\dots, x_n]$ vanishes on the zero locus of a left ideal $I$ of $\mathbb{H}[x_1,\dots, x_n]$  iff for every $b\in\mathbb{H}$ with $b^2=-1$, there exists a natural number $N\geq 1$ such that 
			\[
			(bF)^N\in I+I(bF)+I(bF)^2+\cdots+I(bF)^N.
			\] 
		\end{corollary}
		\begin{proof}
			We modify the proof of Theorem \ref{thm: Strong Nullstellensatz over quaternions} to show this result. We only need to prove the ``if" direction. The only place in the proof of this direction  that needs to be modified is the step at which we set  $a=F(a_1,...,a_n)^{-1}$. We note that  there exists an element $b\in\mathbb{H}$ with $b^2=-1$ such that $bF(a_1,...,a_n)$ commutes with  $a_1,...,a_n$. Replacing $a$ by $b$ in the rest of the proof of Theorem \ref{thm: Strong Nullstellensatz over quaternions} gives a proof of the weakened version. 
		\end{proof}
		The following example shows that the condition cannot be weakened further in this fashion. 
		\begin{example}
			Consider the left ideal $I$ of $\mathbb{H}[x]$ that is generated by $x-i$. We have $\mathcal{V}(I)=\{i\}$. The constant polynomial $F=1$ does not vanish on the zero locus of $I$. However, we have
			\[
				bF=-\left( b(bi-ib)^{-1}b\right) (x-i)+\left( b(bi-ib)^{-1}\right) (x-i)(bF)\in I+I(bF),
			\]  
			for all $b\in\mathbb{H}$ with $b^2=-1$ except $b=\pm i$. 
 		\end{example}
		Next we discuss a consequence of Theorem \ref{thm: Strong Nullstellensatz over quaternions} regarding the concept of left radical. Recall that the left radical of a left ideal $I$ of a ring $R$ is defined to be the intersection of all completely prime left ideals of $R$ containing $I$ (see \cite[Definition 4.5]{AlonNullestellensatz}).
		It is shown in \cite{AlonNullestellensatz} that a polynomial $F\in \mathbb{H}[x_1,\dots, x_n]$ vanishes on the zero locus of a left ideal $I$ of $\mathbb{H}[x_1,\dots, x_n]$ iff  $F$ belongs to the left radical of $I$  (this is the strong form of Nullstellensatz  given in \cite[Theorem 1.2]{AlonNullestellensatz}).  Our version of Hilbert's Nullstellensatz gives us the following description of the left radical of a left ideal of $\mathbb{H}[x_1,\dots, x_n]$.
		\begin{corollary}
			Let $I$ be a proper left ideal of $\mathbb{H}[x_1,\dots, x_n]$. The left radical of $I$ coincides with the set of all polynomials $F\in \mathbb{H}[x_1,\dots, x_n]$ such that for every $a\in\mathbb{H}$, there exists a natural number $N\geq 1$ satisfying the condition 
			\[
			(aF)^N\in I+I(aF)+I(aF)^2+\cdots+I(aF)^N.
			\] 
		\end{corollary}

	\end{subsection}

\end{section} 

\bibliographystyle{plain}
\bibliography{StrongNullestellensatzOverHbiblan}

 \end{document}